\documentclass[12pt]{article}

\usepackage{amssymb}
\usepackage{amsthm}
\usepackage{amsmath}

\usepackage[pagebackref,hypertexnames=false, colorlinks, citecolor=black, linkcolor=blue, urlcolor=red]{hyperref}

\usepackage[backrefs]{amsrefs}

\usepackage{latexsym}
\usepackage{amssymb}
\usepackage{euscript}

\def\mcc{M\raise.5ex\hbox{c}C}
\def\mccarthy{M\raise.5ex\hbox{c}Carthy}

\def\eg{{\it e.g. }}
\def\ie{{\it i.e. }}

\def\h{{\cal H}}

\def\M{{\cal M}}

\def\m{Mult}

\def\La{\Lambda}

\def\l{\lambda}

\def\vare{\varepsilon}

\let\i=\infty

\def\={\ = \ }

\def\A{{\cal A}}

\def\F{{\cal F}}

\def\C{\mathbb C}

\def\D{\mathbb D}

\def\inn{\ \in \ }

\def\be{\setcounter{equation}{\value{theorem}} \begin{equation}}
\def\ee{\end{equation} \addtocounter{theorem}{1}}
\def\beq{\begin{eqnarray*}}
\def\eeq{\end{eqnarray*}}
\def\se{\setcounter{equation}{\value{theorem}}} 
\def\att{\addtocounter{theorem}{1}}

\def\bp{{\sc Proof: }}
\def\ep{{}{\hfill $\Box$} \vskip 5pt \par}

\def\bl{\begin{lemma}}
\def\el{\end{lemma}}
\def\bt{\begin{theorem}}
\def\et{\end{theorem}}
\def\bprop{\begin{prop}}
\def\eprop{\end{prop}}
\def\bd{\begin{definition}}
\def\ed{\end{definition}}
\def\br{\begin{remark}}
\def\er{\end{remark}}
\def\bexer{\begin{exercise}}
\def\eexer{\end{exercise}}
\def\bfig{\begin{figure}}
\def\efig{\end{figure}}

\newtheorem{theorem}{Theorem}[section]
\newtheorem{prop}[theorem]{Proposition}
\newtheorem{lemma}[theorem]{Lemma}
\newtheorem{cor}[theorem]{Corollary}

\newtheorem{question}[theorem]{Question}

\newtheorem{claim}[theorem]{Claim}

\theoremstyle{definition}
\newtheorem{example}[theorem]{Example}
\newtheorem{remark}[theorem]{Remark}
\newtheorem{definition}[theorem]{Definition}

\title{Operator theory  and the Oka extension theorem}
\author{Jim Agler
\thanks{Partially supported by National Science Foundation Grant
DMS 1068830}
\and
John E. M\raise.5ex\hbox{c}Carthy
\thanks{Partially supported by National Science Foundation Grant DMS 0966845}
}

\renewcommand{\epsilon}{\varepsilon}

\def\norm#1{\| #1 \|}

\def\normdel#1{{\| #1 \|}_\delta}

\def\normdelgen#1{\| #1 \|_{\delta,{\rm gen}}}
\def\normtdelgen#1{\| #1 \|_{t \delta,{\rm gen}}}
\def\normgamgen#1{\| #1 \|_{\gamma,{\rm gen}}}
\def\norminf#1{{\| #1 \|}_\infty}

\def\normm#1{{\| #1 \|}_m}

\def\set#1#2{\{ #1 \, | \, #2\}}

\def\hinf{H^\infty}

\def\hinfdel{H_\delta^\infty}

\def\hinfm{H_{\bf m}^\infty}
\def\hinfdg{H^\infty_{\delta,{\rm gen}}(E)}
\def\gl{{\mathcal G}_\Lambda}
\def\h{\mathcal{H}}
\def\m{\mathcal{M}}

\def\gdel{G_\delta}

\def\kdel{K_\delta}
\def\gtdel{G_{t\delta}}

\def\cdel{{\mathcal{C}}_\delta}

\def\fdel{{\mathcal{F}}_\delta}
\def\fm{{\mathcal{F}}_m}

\def\ftdelgen{{\mathcal{F}}_{t\delta, {\rm gen}}}

\def\l{\lambda}
\def\d{\delta}
\def\s{\sigma}
\newcommand\bem{\left( \begin{array}}
\newcommand\enm{\end{array} \right)}
\renewcommand\L{{\mathcal L}}
\renewcommand{\O}{\Omega}
\def\fdl{{\mathcal{F}}_{\delta,\Lambda}}
\def\fdg{{\mathcal{F}}_{\delta,{\rm gen}}}

\newcommand\clhd{C{\mathcal L}(\h)^d}
\newcommand\clhnd{C{\mathcal L}(\h_n)^d}
\newcommand\hidg{H^\i_{\d,{\rm gen}}}
\newcommand\hidge{H^\i_{\d,{\rm gen}}(E)}

\def\fdellam{\fdl}
\def\glam{{\mathcal G}_\Lambda}
\def\fdelgen{{\mathcal F}_{\delta, {\rm gen}}}
\def\fgamgen{{\mathcal F}_{\gamma, {\rm gen}}}
\def\dom{\rm dom\ }
\def\ball{\rm ball\ }
\def\hinfdelgen{H^\infty_{\delta,{\rm gen}}}
\def\hinftdelgen{H^\infty_{t \delta,{\rm gen}}}
\def\hinfgamgen{H^\infty_{\gamma,{\rm gen}}}
\def\ztof{f}

\begin{document}

\bibliographystyle{plain}

\maketitle
\begin{abstract}
For $\delta$ an $m$-tuple of analytic functions, we define an algebra $\hidg$, 
contained in the bounded analytic
functions on the analytic polyhedron $ \{ \| \delta^l(z) \| < 1, \ 1 \leq l \leq m \}$,
and prove a representation formula for it.
We give conditions whereby every function that is analytic on a neighborhood of
$ \{ \| \delta^l(z) \| \leq 1, \ 1 \leq l \leq m \}$
  is actually in $\hidg$. We use this to give a proof of the Oka extension theorem with bounds.
We define an $\hidg$ functional calculus for operators.

\end{abstract}


\section{Introduction}

In \cite{amy12b} a new proof of the Oka Extension Theorem for p-polyhedra was given using operator-theoretic methods. While the proof used both the functional calculus for commuting operator-tuples and the Oka-Weil approximation theorem for analytic functions defined on a neighborhood of p-polyhedra, it had the novel feature that it revealed norms for which extensions were obtained with precise bounds.

In this paper we shall improve upon the results from \cite{amy12b} in a number of ways. First, we shall eliminate our reliance on the functional calculus for general operator-tuples developed by J. Taylor, and in its place, use a simple functional calculus for commuting diagonalizable matrix-tuples together with a technical axiom described in
Section 3 of the paper. Secondly, we shall not require the Oka-Weil approximation theorem. Rather, using our technical axiom we will shall give an original proof of  the stronger Oka Extension Theorem \cite{oka36}
 that is entirely elementary in nature. Finally, rather than working on p-polyhedra, we shall work with analytic polyhedra defined in general domains of holomorphy.

Once the program described in the above paragraph is carried out, the door is opened for generalization of Oka and Cartan like results in several complex variables that involve the existence of extensions of holomorphic functions. We illustrate this point by giving a  proof of 
a special case of the Cartan Extension Theorem
for holomorphic functions defined on analytic varieties in domains of holomorphy.

%

Here is a prototypical example.
\begin{example}
\label{exa1}
Let $d = 2$, 
 and $\A = H^\i(\D^2)$, the algebra of all bounded analytic functions on $\D^2$.
It was shown in \cite{ag90} that a function $\phi$ is in $H^\i(\D^2)$ and
has norm less than or equal to $ 1$ if and only if there is a graded Hilbert
space $\L = \L^1 \oplus \L^2$ and a unitary operator $V : \C \oplus \L \to \C \oplus \L$, which can be written in block form as
\[
V \= \bordermatrix{ ~& \C &\L \cr
\C &A&B\cr
\L & C & D \cr
},
\]
so that, writing $P^1$ for the projection from $\L$ onto $\L^1$ and
$P^2$ for the projection from $\L$ onto $\L^2$, we have
\be
\label{eqa2}
 \phi(z) \= A + B (z^1 P^1 + z^2 P^2) \left[
I_\L -  (z^1 P^1 + z^2 P^2) D \right]^{-1} C .
\ee
If now $T = (T^1, T^2)$ is a pair of commuting operators on a Hilbert space $\h$, one can modify (\ref{eqa2}) to
\begin{eqnarray}
\att
\nonumber
 \phi(T) \= &
 I_\h \otimes  A + I_\h \otimes B (T^1 \otimes P^1 + T^2 \otimes P^2) \\
& \left[
I_{\h \otimes \L} -  (T^1\otimes P^1 + T^2 \otimes P^2) I_\h \otimes D \right]^{-1}I_\h \otimes C .
\label{eqa3}
\end{eqnarray}
Formula (\ref{eqa3}) makes sense as long as the spectrum of the operator
$ (T^1\otimes P^1 + T^2 \otimes P^2) I_\h \otimes D$ does not contain the
point $1$. This will be guaranteed if $T^1$ and $T^2$ both have spectrum in the open disk $\D$. If, in addition, one assumes that they are both strict contractions (\ie have norm less than one), then it can be shown that the formula (\ref{eqa3}) gives a 
contractive $H^\i(\D^2)$  functional calculus for $T$. Indeed, the contractivity
follows from calculating $I - \phi(T)^* \phi(T)$ and observing that it is positve; the fact that the map is an algebra homomorphism follows from observing that in both (\ref{eqa2}) and
(\ref{eqa3}), if one expands the inverse in a Neumann series, one gets an absolutely convergent
 series of polynomials, provided $\| z \|$ (respectively, $\| T \|$) is less than one.
\end{example}

In \cite{amy12b},  Example~\ref{exa1} was generalized to polynomial polyhedra as follows.
 Let $\delta = (\d^1,\dots,\d^m) $ be a polynomial mapping,
which we think of  as a map into $\C^m$ with the $\ell^\i$-norm.
Let 
\beq
\gdel & \= & \{ z \in \C^d\, : \, \| \d (z) \| < 1 \} ,\\
\kdel & \= & \{ z \in \C^d\, : \, \| \d (z) \| \leq 1 \},
\eeq
and for any domain $\O \subset \C^d$ let $O(\O)$ denote the algebra of holomorphic
functions on $\O$.
Let us use $\clhd$ to denote the set of all commuting $d$-tuples of bounded linear operators on 
the Hilbert space $\h$.
Define the family of operators $\fdel$ 
 by
\be
\label{eqa111}
\fdel \= \{ T \in \clhd \, : \, 
\| \d(T) \| \leq 1 \}.
\ee
Then define an algebra $\hinfdel$ of holomorphic functions on $\gdel$ by
\be
\label{eqa11}
\hinfdel \= \{ \phi \in O(\gdel) \, : \, \sup \{ \| \phi(T) \| \, : \, T \in \fdel \ {\rm and\ }\sigma(T) \subset \gdel \} < \i \}.
\ee
There is a representation formula for functions in $\hinfdel$ 
very similar to (\ref{eqa2}); this generalizes to operators just as in (\ref{eqa3}).
It was shown in \cite{amy12b} that this leads not only to a new proof of the Oka extension theorem
\cite{oka36},
but that it gives sharp bounds for the norms of the extensions, which the function theory approaches do not yield.

\subsection{Main results}

The definition of  $\hinfdel$, and also of $\fdel$ if $\delta$ is allowed to be a holomorphic mapping on
some domain,  has the drawback that it requires knowledge of the
Taylor spectra of the tuples $T$.
In Section~\ref{secb} we
define versions of these objects in a more direct and elementary way that
avoids Taylor's  theory. This goal is accomplished through
restricting our use of operators in defining norms to the diagonalizable, finite dimensional case.
For commuting diagonalizable matrices there is no mystery to the functional calculus:
one simply chooses a basis of eigenvectors, and applies the function to the joint eigenvalue.

In Section \ref{secb}, we describe a set of matrices $\fdg \subsetneq \fdel$,
and  algebras $\hidg$, and a localized version $\hidge$, that are substitutes for $\hinfdel$.
In Section~\ref{secc} we prove a representation theorem for $\hidg$ and $\hidge$.

If $U$ is a domain of holomorphy, and $\d \in O(U)^m$ is no longer assumed to be 
polynomial, the formal calculations work just as before, leading to an algebra 
$\hidg$ of functions on $\gdel = \{ z \in U : \| \d (z) \| \leq 1 \}$.
To be useful, we would like to know what functions are in $\hidg$.
In particular, we define $\d$ to be utile (Definition~\ref{def4.22} below) if every
function holomorphic on a neighborhood of $\kdel$ is in $\hidg$.

In Section~\ref{secd}, we prove two theorems that give conditions for
$\d$ to be utile in terms of operator theoretic properties of
the set $\fdg$. These conditions are not always satisfied. However, we show in Section~\ref{secsubord} that one can add more functions to the set $\delta$ to get a new family
$\gamma$ so that $K_\gamma = \kdel$, and hence the function theory is unchanged, but so that
$\gamma$ is utile, and therefore one can attack function theory problems using these operator theory tools.

To illustrate the use of these ideas, 
in Section~\ref{sece} we  prove refinements of the Oka extension theorem and the Oka-Weil theorem. In Section~\ref{sech} we prove a special case of the Cartan extension theorem.

\subsection{Functional Calculus}

If $\A$ is a topological algebra that contains the polynomials $\C[z^1, \dots, z^d]$, and 
if $T = (T^1, \dots, T^d)$ is a $d$-tuple of commuting operators on a Banach space~$\mathcal X$,
an {\em $\A$-functional calculus for $T$} is a
continuous homomorphism $\pi$ from  $\A$ to $B({\mathcal X})$
such that $\pi(z^l) = T^l$ for each $1 \leq l \leq d$.

In the ground-breaking papers \cite{tay70b,tay70a}, J.L. Taylor developed a functional calculus
for the algebras $O(U)$, where $O(U)$ denotes the algebra of all holomorphic functions on 
the open set $U \subseteq \C^d$. This was a multi-variable version of the Riesz-Dunford calculus
in one variable. He also defined a spectrum for commuting $d$-tuples, now called the Taylor spectrum, and showed that for any commuting $d$-tuple $T$, there was an  $O(U)$ functional calculus for $T$
if and only if the Taylor spectrum of $T$ lay in $U$.

A drawback to Taylor's approach is that it is often difficult to determine what the spectrum of $T$ is,
and the formulas defining $\phi(T)$ can be too complicated to handle. (We shall routinely write $\phi(T)$ for $\pi(\phi)$).

 In Section~\ref{secfc} we propose a different approach.
The key idea is that
we choose some basic functions $\delta^l$, and then define an $\hidg$
functional calculus by writing down a relatively simple formula
for $\phi(T)$ in terms of $\delta(T)$, that gives a well-defined operator whenever $\| \delta (T) \| < 1$.

The formula (a modification of (\ref{eqa3}))
 can be thought of as a multi-variable version of the Sz.-Nagy-Foia\c{s}
functional calculus, which is an $H^\i(\D)$ functional calculus for (single)
Hilbert space contractions  that have spectrum in the open unit disk $\D$, or, more generally, have no unitary summand \cite{szn-foi}. Our approach generalizes ideas of C.-G.  Ambrozie and D. Timotin
\cite{at03}, of J.A. Ball and V. Bolotnikov \cite{babo04}, and of the authors and N.J. Young
\cite{amy12b}.

Our approach leads in Section~\ref{secfc}  to  a functional calculus for any $p$-polyhedron
(a set of the form $\kdel$), or, more generally, any rationally convex set. The rational Oka-Weil theorem asserts that any holomorphic function
on a neighborhood of a rationally convex set is a limit of rational functions, so at first blush
the existence of the functional calculus is unremarkable. However, our functional calculus
comes with a {\em norm control}, and this is its key property.
In Theorem~\ref{thmfc2}, we show there is no ambiguity: when $\phi(T)$ is defined both by our definition and
by Taylor's, the two definitions coincide.


\section{ $\hinfdg$}
\label{secb}

Let $U$ be a domain of holomorphy in $\C^d$, fixed
throughout the paper.
We will take $\d = (\d_1,\dots, \d_m) \inn O(U)^m$ to be an $m$-tuple of non-constant
holomorphic functions on $U$ and think of it as a map into $\C^m$ with the $\ell^\i$-norm.
(More generally, one can consider $\d$ as a map into the $r$-by-$s$ matrices. This is natural
if one wishes to define a functional calculus not for $d$-tuples $T$ of commuting contractions,
but defined by other inequalities, for example if $T$ is a row-contraction, defined
by $T^1 T^{1*} + \dots + T^d T^{d*} \leq I$. For convenience
we will not do our calculations in this generality).

We define two analytic polyhedra in $U$ by 
$\gdel = \d^{-1} (\D^m)$ and $\kdel = \d^{-1}(\overline{\D ^m})$.
{\em We shall always assume that $\kdel$ is compact.}
If $T$ is a $d$-tuple of pairwise commuting operators acting on an $N$ 
dimensional Hilbert space, we say that $T$ is {\em generic} if there exist $N$ linearly
independent joint eigenvectors whose corresponding joint eigenvalues are
 distinct (equivalently, $T$ is diagonalizable and $\s(T)$ has cardinality $N$).
 If
$\La = \{ \l_1, \dots, \l_N \}$ is a set consisting of exactly $N$ distinct points in $\C^d$, we
let $\gl$ denote the set of generic $d$-tuples of pairwise commuting operators $T$
acting on $\C^N$ with $\s(T) = \La$.
We shall let superscripts identify components of $d$-tuples and subscripts
identify components of $N$-tuples. Thus, if  $T = (T^1, \dots, T^d) \in \gl$, then there
exist $N$ vectors $k_1, \dots, k_N \in \C^N$  such that
\[
T^r k_j \= \l^r_j k_j 
\]
for $1 \leq r \leq d$ and $1 \leq j \leq N$. 
If $T \in \gl$ and $f$ is holomorphic on a
neighborhood of $\s(T)$, then we define  $f(T)$ to be the unique operator on
$\C^N$ that satisfies
\be
\label{eqb2}
f(T) k_j \= f(\l_j) k_j, \qquad 1 \leq j \leq N .
\ee

We now define generic and local versions of the sets $\fdel$, $\hinfdel$ and the norm
 $\normdel{\cdot}$ from \cite{amy12b}. Fix a set $E \subseteq \gdel$ and let $E^\C$ denote the space of complex valued functions on $E$. For finite $\Lambda \subset \gdel$, let
 $\fdellam$ be the set in $\glam$ defined by
\be \label{eq2.20}
\fdellam = \set{T \in \glam}{\norm{{\delta}_l(T)} \le 1, l=1,\ldots,m}.
\ee
We then define a collection of generic operators, $\fdelgen$, by setting
\be \label{eq2.25}
\fdelgen = \bigcup_{\substack{\Lambda \text{finite}\\ \Lambda \subset \gdel}}\fdellam.
\ee
For $\ztof \in E^\C$, define $\normdelgen{\ztof}$ by the formula,
\be\label{eq2.30}
\normdelgen{\ztof} = \sup_{\substack{T \in \fdelgen\\ \sigma(T) \subseteq E}} \norm{\ztof(T)}
\ee
Note that the notation, $\normdelgen{\ztof}$ sees $E$ due to the fact that $E = \dom \ztof$. Finally, we define $\hinfdelgen(E)$ to consist of all $\ztof \in E^\C$ such that $\normdelgen{\ztof}$ is finite.

Before continuing, we remark that the space $\hinfdelgen(E)$, as just defined, refines the space $\hinfdel$ considered in \cite{amy12b} in two ways. First, by suping over generic operators rather than operators with spectrum in $\gdel$, use of the Taylor functional calculus is avoided (though {\em a priori} it may be the case that $\normdelgen{\cdot} \ne \normdel{\cdot}$).
 Secondly, we have localized the definition of $\hinfdel$ through the introduction of the set $E$.
 This represents a considerably more general scheme than considered in \cite{amy12b}. In particular, our work here will make it clear that many of the proofs in \cite{amy12b} do not require that the elements of $\hinfdel$ be holomorphic. However, when $E=\gdel$ and $\ztof=\phi$ is holomorphic on $\gdel$ it turns out that that $\normdelgen{\ztof} = \normdel{\phi}$ --- see Theorem~\ref{thm3.10}.

The following proposition is a straightforward analog of Proposition 2.4 in \cite{amy12b}. For $E$ a set, we let $\ell^\infty(E)$ denote the Banach algebra of bounded complex valued functions $\ztof$ on $E$ equipped with the norm,
$$\norminf{\ztof} = \sup_{e \in E}|\ztof(e)|.$$
\begin{prop}\label{prop2.10}
$\hinfdelgen(E)$ equipped with $\normdelgen{\cdot}$ is a Banach algebra. Furthermore, $\hinfdelgen(E) \subseteq \ell^\infty(E)$ and $\norminf{\ztof} \le \normdelgen{\ztof}$ for all $\ztof \in \hinfdelgen$.
\end{prop}
\begin{proof}
The only assertion that is not immediately obvious is that $\hinfdelgen(E)$ is complete.
Let $f_n$ be a Cauchy sequence in $\hinfdelgen(E)$. It converges in $\ell^\i(E)$ to a function $f$.
For any $T$ in $\fdelgen$ with $\sigma(T) \subseteq E$, we have
$f_n(T) \to f(T)$. Therefore
\[
\| f(T) \| \ \leq \ \sup_n \| f_n(T) \| \leq \sup_n \normdelgen{ f_n} ,
\]
so $f$ is in $\hinfdelgen(E)$ as required.
\end{proof}
Further useful regularity of the elements of $\hinfdelgen(E)$ is revealed in the following proposition. For $\alpha_1, \alpha_2 \in \D$, we recall that the \emph{Caratheodory pseudometric}, $d$, is defined on $\D$ by the formula,
$$d(\alpha_1,\alpha_2) = \Big|\frac{\alpha_2-\alpha_1}{1-\overline{\alpha_1}\alpha_2}\Big|,\ \ \alpha_1,\alpha_2 \in \D.$$
\begin{prop}\label{prop2.20}
If $\ztof \in \ball \hinfdelgen(E)$ and $\lambda_1,\lambda_2 \in E$, then
\be\label{eq2.38}
d(\ztof(\lambda_1),\ztof(\lambda_2)) \le \max_{1 \le l \le m}d(\delta_l(\lambda_1),\delta_l(\lambda_2))
\ee
\end{prop}
\begin{proof}
Fix $\lambda_1,\lambda_2 \in E$ with $\lambda_1 \ne \lambda_2$ and set $\Lambda=\{\lambda_1,\lambda_2\}$. For $T \in \glam$, let $\theta_T$ denote the angle between the eigenspaces of $T$. By the remark following (2.7) in \cite{ag91}, elements of $\glam$ are determined uniquely by this angle between their eigenspaces. Furthermore, by the calculations
 immediately following (2.13) in \cite{ag91}, if $g$ is a function defined on $\Lambda$ and $T \in \glam$, then
\be\label{eq2.40}
\norm{g(T)} \le 1 \iff \sin\theta_T \ge d(g(\lambda_1), g(\lambda_2)).
\ee
Recalling \eqref{eq2.20}, we see that \eqref{eq2.40} implies that if $T \in \glam$, then
\be\label{eq2.50}
T \in \fdellam \iff \sin \theta_T \ge \max_{1 \le l \le m}d(\delta_l(\lambda_1),\delta_l(\lambda_2)).\notag
\ee
In particular, there exists $T \in \fdellam$ such that
\be\label{eq2.60}
\sin \theta_T = \max_{1 \le l \le m}d(\delta_l(\lambda_1),\delta_l(\lambda_2)).\notag
\ee
As for this $T$ $\norm{\ztof(T)} \le 1$, we deduce from \eqref{eq2.40} that \eqref{eq2.38} holds.
\end{proof}
\begin{cor}\label{cor2.10}
If $\ztof \in \hinfdelgen(E)$, then $\ztof$ is continuous on $E$.
\end{cor}
\begin{proof}
If $\ztof\ne 0$ in $\hinfdelgen(E)$, then Propositon \ref{prop2.20} implies that $\normdelgen{\ztof}^{-1}\ztof$ is continuous. 
\end{proof}
The equivalence of conditions (a), (b), and (c) in the next proposition is an exact analog of Proposition 2.7 from \cite{amy12b}.
\begin{prop}\label{prop2.30}
Let $E \subset \gdel$ be finite. The following are equivalent.
\begin{align*}
&(a)\ \hinfdelgen(E) = \ell^\infty(E).\\
&(b)\ \lambda^r|_E \in \hinfdelgen(E) \text{ for } r=1,\ldots,d. \quad \quad \quad \quad \\
&(c)\  {\mathcal F}_{\delta,E} \text{ is bounded}.\\
&(d)\ \delta|_E \text{ is } 1-1.
\end{align*}
\end{prop}
\bp
It is obvious that {\em(a)} implies {\em(b)} and that {\em(b)} implies {\em (c)}.

$(c) \Rightarrow (d)$: Let $E = \{ e_1, \dots, e_N \}$, and assume that 
$\d(e_1) = \d(e_2)$.
Choose vectors $k_1, \dots, k_N$ so that $\{ k_2, \dots, k_N\}$ are orthonormal, and
$k_1$ is orthogonal to $\{k_3,\dots, k_N\}$, but at angle $\theta$ to $k_2$.
Then $T$ defined by $T^r k_j = e^r_j k_j$ is in ${\mathcal F}_{\delta,E} $, but
the norm of $T$ will tend to infinity as $\theta$ tends to $0$.

$(d) \Rightarrow (a)$: It is sufficient to prove that for each $1 \leq j \leq N$, the
function $f_j$ defined by $$
f_j(e_i) = 0, i \neq j, \qquad f_j(e_j) = 1$$
is in  $\hinfdelgen(E)$.
Since $\delta$ is one-to-one, there exists some function $g_j:\d(E) \to \C$ such that
$g_j \circ \d = f_j$. To prove that $f_j$ is in  $\hinfdelgen(E)$, it suffices to prove that
$g_j$ is bounded on $m$-tuples of commuting generic contractions.
This latter assertion follows from Lemma~\ref{lemb5}.
\ep

\bl
\label{lemb5} Let $\Lambda = \{ \l_1, \dots, \l_N \}$ be a set of $N$ distinct points in $\D^m$.
The function $g_j$, defined on $\Lambda$ by $g_j(\l_i) = \delta_{ij}$, is  bounded
on $m$-tuples of commuting contractions.
\el
\bp
For each $i \neq j$, choose $r_i$ so that $\l_j^{r_i} \neq \l_i^{r_i}$.
Define $h$ by 
\[
h(\l) \=
\prod_{i\neq j} \frac{\l^{r_i} - \l_i^{r_i} }{1 - \overline{ \l_i^{r_i}} \l^{r_i}} .
\]
By von Neumann's inequality, each factor in $h(T)$ has norm at most one if $T$ is an $m$-tuple of
commuting contractions. Let $g_j = \frac{1}{h(\l_j)} h$.
\ep


\section{Representing elements of $\hinfdg$}
\label{secc}
For ${\mathcal X}$ a Banach space, we set $\ball {\mathcal X} = \set{x \in {\mathcal X}}{\norm{x} \le 1}$.
In this section we shall derive three conditions that are necessary and sufficient for a 
given $\ztof \in E^\C$ to be an element of $\ball \hinfdelgen(E)$. 
These conditions, which are summarized in Theorem \ref{thm3.10} below,
 represent the appropriate extension of Theorem 4.5 from \cite{amy12b} to the $\hinfdelgen(E)$ setting. To 
state the theorem will require a number of definitions.
\bd
\label{def3.10}
Let $\ztof$ be a function on $E$. We say a $4$-tuple $(a,\beta,\gamma,D)$ is a $(\delta,E)$-realization for $\ztof$ if $a \in \mathbb{C}$ and there exists a decomposed Hilbert space, $\m = \oplus_{l=1}^m \m_l$, such that\\ \\
(i) the $2\times2$ matrix,
$$V=\begin{bmatrix}a&1\otimes\beta\\\gamma\otimes1&D\end{bmatrix},$$
acts isometrically on $\mathbb{C}\oplus\m$,\\ \\
(ii) for each $\lambda \in E$, $\delta(\lambda)$ acts on $\m$ via the formula,
\be
\label{eqc16}
\delta(\lambda)(\oplus_{l=1}^m x_l)=\oplus_{l=1}^m \delta_l(\lambda)x_l, 
\ee
(iii) for each $\lambda \in E$,
\be
\label{eqc17}
\ztof(\lambda)=a+<\delta(\lambda){(1-D\delta(\lambda))}^{-1}\gamma,\beta>.
\ee
\ed
\bd
\label{def3.20}
For $E\subseteq \gdel$, we let $\cdel(E)$ denote the collection of functions $h$ defined on $E \times E$ that have the form,
$$h(\lambda,\mu)=\sum_{l=1}^m \big(1-\overline{\delta_l(\mu)}\delta_l(\lambda)\big)a_l(\lambda,\mu),\ \ \lambda,\mu \in E,$$
where for $l=1,\ldots,m$, $a_l(\lambda,\mu)$ is positive definite on $E$.
\ed
\bl
\label{lem3.10}
If $E\subset \gdel$ is finite and $\ztof \in \ball \hinfdelgen(E)$, then $1-\overline{\ztof(\mu)}\ztof(\lambda) \in \cdel(E)$.
\el
\bp This follows from a Hahn-Banach separation argument, as in \cite{ampi}*{Section 11.1}.
\ep

\begin{definition}\label{def3.30}
Let $m$ be a positive integer and $\h$ a separable infinite dimensional Hilbert space. We let $\fm$ denote the collection of pairwise commuting $m$-tuples of contractions acting on $\h$. If $F\in O(\D^m)$, we define $\normm{F}$ by
\be\label{eq3.10}
\normm{F}=\sup_{\substack{T\in \fm\\ \sigma(T) \subset \D^m}}\norm{F(T)}.
\ee
We let $\hinfm$ denote the set of all $F\in O(\D^m)$ such that $\normm{F} < \infty$.
\end{definition}
\begin{prop}\label{prop3.10}
$\hinfm$ equipped with $\normm{\cdot}$ is a Banach algebra. Furthermore, $\hinfm \subseteq \hinf$ and if $F \in \hinfm$, then $\sup_{z \in \D^m}|F(z)| \le \normm{F}$.
\end{prop}
\bp This is proved for example in \cite[Prop. 2.4]{amy12b}.
\ep

\begin{lemma}\label{lem3.20}
If $\{F_n\}$ is a sequence in $\ball{\hinfm}$, $F\in O(\D^m)$ and $F_n \to F$ in $O(\D^m)$, then $F\in \ball \hinfm$.
\end{lemma}
\begin{proof}
If $F_n \to F$ and $\sigma(T) \subset \D^m$, then $F_n(T) \to F(T)$ in operator norm. Hence, if $T\in \fm$ and $\sigma(T) \subset \D^m$,
$$\norm{F(T)} = \lim_{n \to \infty}\norm{F_n(T)} \le 1.$$
Therefore $\normm{F} \le 1$.
\end{proof}
\begin{lemma}\label{lem3.30}
If $F \in \ball \hinfm$, then there exist $m$ positive definite functions on $\D^m$, $A_1,\ldots,A_m$, such that
$$1-\overline{F(w)}F(z) = \sum_{l=1}^m (1-\overline{w_l}z_l)A_l(z,w)$$
for all $z,w \in \D^m$.
\end{lemma}
\bp This was proved in \cite{ag90}.
\ep
\begin{theorem}\label{thm3.10}
Let $E \subseteq \gdel$ and let $\ztof \in E^\C$. The following conditions are equivalent.
\begin{align*}
&(a)\quad \ztof \in \ball \hinfdelgen(E).\\
&(b)\quad [1-\overline{\ztof(\mu)}\ztof(\lambda)] \in \cdel(E).\\
&(c)\quad \ztof \text{ has a } (\delta ,E)-realization.\\
&(d)\quad \exists_{F \in \ball \hinfm} \ \forall_{\lambda \in E}\ \ \ztof(\lambda)=F(\delta(\lambda)).
\end{align*}
\end{theorem}
\begin{proof}
That (b) implies (c) follows from a Lurking Isometry argument. To see that (c) implies (d), let $(a,\beta,\gamma,D)$ be a realization for $\ztof$ as in Definition \ref{def3.10}. If we define $F$ by the formula,
\be
\label{eqc23}
 F(Z)=a+<Z{(1-DZ)}^{-1}\gamma,\beta>,
\ee
then $F \in \ball \hinfm$ and $\ztof=F\circ\delta|E$. To see that (d) implies (a), fix $F \in \ball \hinfm$ and assume that $\ztof=F\circ\delta|E$. If $\Lambda$ is a finite subset of $E$ and $T \in \fdellam$, then $\delta(T) \in \fm$. Hence,
$$\norm{\ztof(T)} = \norm{F(\delta(T))} \le 1.$$

We now turn to the task of proving that (a) implies (b). 
Assume that $\ztof \in \ball \hinfdelgen(E)$. Let $\{\lambda_k\}$ be 
a dense sequence in $E$, and  for $n\ge1$, set $\Lambda_n = \{\lambda_1,\ldots,\lambda_n\}$. Fix $n$. As $\ztof \in \ball \hinfdelgen(E)$, $\ztof \in \ball \hinfdelgen(\Lambda_n)$. Hence, by Lemma \ref{lem3.10}, $[1-\overline{\ztof(\mu)}\ztof(\lambda)] \in \cdel(\Lambda_n)$. As (b) implies (c) and (c) implies (d), there exists $F_n \in \ball \hinfm$ such that $\ztof|\Lambda_n = (F_n \circ \delta)|\Lambda_n$.

By Proposition \ref{prop3.10}, $\{F_n\}$ is uniformly bounded on $\D^m$.
 Hence, there exist $F\in O(\D^m)$ and a subsequence $\{F_{n_i}\}$ such that $\{F_{n_i}\} \to F$ as $i \to \infty$ in $O(\D^m)$.
 By Lemma \ref{lem3.20}, $F \in \ball \hinfm$. Furthermore, as $\ztof|\Lambda_n = (F_n \circ \delta)|\Lambda_n$ for each $n$, $\ztof(\lambda_k) = F(\delta(\lambda_k))$ for all $k$. But $\{\lambda_k\}$ is dense in $E$, so that, by Corollary \ref{cor2.10}, $\ztof(\lambda) = F(\delta(\lambda))$ for all $\lambda \in E$. Consequently, if $A_1,\ldots,A_m$ are as in Lemma \ref{lem3.30},
$$ 1-\overline{\ztof(\mu)}\ztof(\lambda)=1-\overline{F(\delta(\mu))}F(\delta(\lambda))=\sum_{l=1}^m (1-\overline{\delta_l(\mu)}\delta_l(\lambda))A_l(\delta(\lambda),\delta(\mu)).$$
As the functions, $a_l(\lambda,\mu)= A_l(\delta(\lambda),\delta(\mu))$, are positive definite on $E$, this proves via Definition \ref{def3.20} that $1-\overline{\ztof(\mu)}\ztof(\lambda) \in \cdel(E)$.
\end{proof}

\section{Utility}
\label{secd}
A fundamental fact, which allows for the purely operator theoretic construction in Section 2 to have significance for several complex variables, is that the elements of $\hinfdelgen(\gdel)$ are holomorphic functions on $\gdel$.
We adopt the abbreviated notation, $\hinfdelgen$, for $\hinfdelgen(\gdel)$.
\begin{prop}\label{prop4.20}
$\hinfdelgen$ equipped with $\normdelgen{\cdot}$ is a Banach algebra. Furthermore, if $H^\infty(G_\delta)$ denotes the space of bounded holomorphic functions on $G_\delta$ equipped with the sup norm, ${\norm{\cdot}}_\infty$, then $\hinfdelgen \subseteq H^\infty(G_\delta)$ and ${\norm{\phi}}_\infty \le \normdelgen{\phi}$ for all $\phi \in \hinfdelgen$.
\end{prop}
\begin{proof}
That $\hinfdelgen$ is a Banach algebra, $\hinfdelgen \subseteq L^\infty(\gdel)$, and $\norminf{\phi} \le \normdelgen{\phi}$ whenever $\phi \in \hinfdelgen$ all follow from Proposition \ref{prop2.10}. That $\hinfdelgen \subseteq O(\D^m)$, follows from (a) implies (d) in Theorem \ref{thm3.10}.
\end{proof}

We now come to one of the central points of the paper. Our goal is to use the operator theoretic construction of Section 2 and the representation result Theorem \ref{thm3.10} to study holomorphic functions on domains of holomorphy. 
If this program is to be carried out, then there must be an ample supply of holomorphic functions in $\hinfdelgen$. Thus, the following definition is germane.
\begin{definition}\label{def4.22}
Let $\delta \in O(U)^m$. We say that $\delta$ is {\em utile} if $f \in \hinfdelgen$ whenever $f$ is holomorphic on a neighborhood of $K_\delta$.
\end{definition}
\subsection{An Operator-theoretic Characterization of Utility}
Theorem \ref{thm4.20} below gives a characterization for when a tuple $\delta$ is utile in terms of the following natural operator theoretic notions.
\begin{definition}\label{def4.30}
Let us agree to say that $\delta$ is {\em bounding} if $\fdelgen$ is bounded, i.e., for each $r=1,\ldots,m$,
$$\sup_{T \in \fdelgen} \norm{T^r} < \infty.$$
\end{definition}
Note that if $\delta$ is bounding and $\{T_n\}$ is a sequence in $\fdelgen$, then $\oplus_{n=1}^\infty T_n$ is a well defined bounded operator.
\begin{definition}\label{def4.40}
We say that $\delta$ is {\em spectrally determining} if $\delta$ is bounding and if $\sigma(\oplus_{n=1}^\infty T_n) \subset \kdel$ for every sequence $\{T_n\}$ in $\fdelgen$.
\end{definition}
As we have adopted the strategy of studying holomorphic functions defined on a neighborhood of $K_\delta$, the following technical modification of the notions of bounding and spectrally determining will prove useful.
\begin{definition}\label{def4.50}

We say that $\delta$ is {\em strictly bounding (resp. strictly spectrally determining)} if there exists $t<1$ such that $t\delta$ is bounding (resp. spectrally determining.
\end{definition}
Note that if $t<1$, then $\fdelgen \subset \ftdelgen$ so that if $\delta$ is strictly bounding, then $\delta$ is bounding. Likewise, if $\delta$ is strictly spectrally determining, then $\delta$ is spectrally determining. To prove this, assume $t\delta$ is spectrally determining and let $\{T_n\}$ be a sequence in $\fdelgen$. As $\fdelgen \subset \ftdelgen$, $T=\oplus_n T_n$ is a well defined operator with $\sigma(T) \subseteq K_{t\delta} \subset U$. Hence, the operators, $\delta_l(T)$, $l=1,\ldots,m$ make sense by the Taylor functional calculus. Furthermore, as $T_n \in \fdelgen$, for each $l=1\ldots,m$,
$$\norm{\delta_l(T)} = \norm{\delta_l(\oplus_n T_n)} = \norm{\oplus_n\delta_l(T_n)} \le 1.$$
Hence, by the spectral mapping theorem, $\sigma(T) \subseteq K_\delta$, as was to be shown.

We used the fact that $\delta_l(\oplus_n T_n) = \oplus_n\delta_l(T_n)$. This follows 
from Lemma~\ref{lemds}.

\bl
\label{lemds} Let $\h_n$ be a sequence of Hilbert spaces, and 
let $T_n$ be in $\clhnd$.  Assume that $T := \oplus_n T_n$ is bounded, and its spectrum lies
inside $U$. Let $g \in O(U)$. Then $g( \oplus_n T_n) = \oplus_n g(T_n)$.
\el
\bp
This result follows from Vasilescu's construction of the Taylor functional calculus
as an integral of an operator-valued Martinelli kernel
(see \cite[Section III.11]{vas82} or \cite{cur88}). 
The integrand for $T$ is the direct sum of the integrands for each $T_n$.
\ep


The authors have been unable to resolve fully the following issue.
\begin{question}\label{ques4.10}
Does bounding imply spectrally determining? 
\end{question}

However, we show in Theorem~\ref{thm4.20} that utile and spectrally determining are equivalent, and in Theorem~\ref{thm6.10} that strictly bounding implies utile.

In some cases  Question \ref{ques4.10} can be answered.
\begin{prop}\label{prop4.30}
Let $K_\delta$ be a p-polyhedron (i.e. $U = \C^d$, and  $\delta_1,\ldots,\delta_m$ are polynomials).
 If $\delta$ is bounding, then $\delta$ is spectrally determining.
\end{prop}
\begin{proof}
Fix a sequence $\{T_n\}$ in $\fdelgen$. As $\delta$ is bounding, $T=\oplus_{n=1}^\infty T_n$ is a well defined bounded operator. Furthermore, if $p$ is a polynomial, $p(T)=\oplus_{n=1}^\infty p(T_n)$. Thus, if $1 \le l \le m$,
$$\norm{\delta_l(T)} = \norm{\oplus_{n=1}^\infty \delta_l(T_n)} \le 1,$$
so that $\sigma(\delta(T)) \subseteq (\D^-)^m$. Hence, by the spectral mapping theorem for polynomials,
$$\delta(\sigma(T)) = \sigma(\delta(T)) \subseteq (\D^-)^m.$$
Thus,
$$\sigma(T) \subseteq \delta^{-1}((\D^-)^m) = K_\delta. 
$$
\end{proof}

We now can state our characterization of utility in operator terms.
\begin{theorem}\label{thm4.20}
Let $\delta \in O(U)^m$. Then $\delta$ is utile if and only if $\delta$ is spectrally determining.
\end{theorem}
The proof of Theorem \ref{thm4.20} will be presented in Subsection~\ref{subsecut}. 
The proof of sufficiency will be a straightforward operator-theoretic proof that relies on the Taylor functional calculus while the proof of necessity will rely on some nontrivial function theory which we discuss in the next subsection.

We record the following simple corollary of Proposition \ref{prop4.30} and Theorem \ref{thm4.20} for future reference.
\begin{prop}\label{prop4.40}
Let $K_\delta$ be a p-polyhedron.
If $\delta$ is bounding, then $\delta$ is utile.
\end{prop}
\subsection{Some Function Theory on Domains of Holomorphy}
In this section we record some results from  function theory on domains of holomorphy that
will be used in the sequel. Recall that $O(U)$ is naturally a Fr\'echet space when endowed with the topology of uniform convergence on compact subsets of $U$. Also recall that if $\mathcal{A}$ is an algebra, then a \emph{complex homomorphism} of $\mathcal{A}$ is a homomorphism defined on $\mathcal{A}$ whose target algebra is $\C$.
\begin{theorem}\label{thm4.30}
$\chi$ is a continuous complex homomorphism of $O(U)$ if and only if there exists $\alpha \in U$ such that $\chi(f) = f(\alpha)$ for all $f \in O(U)$.
\end{theorem}

This follows from \cite[Thm. VII.4.1]{ran86}. 
Another basic fact which we shall require is the following ``baby'' Corona Theorem. The result is well-known, though we cannot find a reference where it is explicitly stated. It can be proved by the
observation that if $\alpha$ is not in $U^-$, then the functions $ R_\alpha^r(\lambda)$ can be chosen to be continuous, because the maximal ideal space of $A(U^-)$ is $U^-$ (see \eg \cite{ros61}). If $\alpha \in \partial U$, take a sequence of points $\alpha_n$ tending to $\alpha$
from outside $U^-$ and use a normal families argument.
\begin{theorem}\label{thm4.40}
If $\alpha \in \C^d \setminus U$, then there exist $R_\alpha^1,\ldots,R_\alpha^d \in O(U)$ such that
$$1 \= \sum_{r=1}^d R_\alpha^r(\lambda) (\lambda^r - \alpha^r)$$
for all $\lambda \in U$.
\end{theorem}

\subsection{The Proof of the Utility Theorem}
\label{subsecut}

This subsection is devoted to a proof of Theorem \ref{thm4.20}. First assume that $\delta$ is spectrally determining. Fix $f$, holomorphic on a neighborhood of $K_\delta$. We need to show that $f \in \hinfdelgen$.
To that end, fix a sequence, $\{T_n\}$ in $\fdelgen$. As $\{T_n\}$ in $\fdelgen$ and $\delta$ is spectrally determining, if we set $T=\oplus_n T_n$, then 
$T$ is a well defined bounded operator with $\sigma(T) \subseteq K_{\delta}$. Hence, as $f$ is holomorphic on a neighborhood of $K_{\delta}$, $f(T)$, as defined by the Taylor functional calculus, is a well defined bounded operator. 
But $f(T) = \oplus_n f(T_n)$ by Lemma~\ref{lemds}, so
  $\norm{f(T_n)} \le \norm{f(T)}$ for all $n$. Summarizing, we have shown that if $\{T_n\}$ is a sequence in $\fdelgen$, then
$\sup_n \norm{f(T_n)} < \infty$. Thus, $\normdelgen{f} < \infty$ and $f \in \hinfdelgen$.

Now assume that $\delta$ is utile. We first show that $\delta$ is bounding.
As $f(\lambda) = \lambda^r$ is holomorphic on a neighborhood of $K_\delta$ and $\delta$ is utile, $f \in \hinfdelgen$. Hence,
$$\sup_{T \in \fdelgen}\norm{T^r} = \sup_{T \in \fdelgen}\norm{f(T)} = \normdelgen{f} < \infty.$$
This proves that $\delta$ is bounding.

To complete the proof that $\delta$ is spectrally determining, fix a sequence, $\{T_n\}$ in $\fdelgen$. As $\delta$ is bounding, $T=\oplus_n T_n$ is a well defined bounded operator. For definiteness, we assume $T$ is in $\clhd$.
We need to show that $\sigma(T) \subseteq K_\delta$.

Since $\delta$ is utile, $O(U) \subset \hinfdelgen$. It follows that the formula,
$$\tau(f) = \oplus_n f(T_n),\ \ \ f \in O(U)$$
defines an algebra homomorphism, $\tau:O(U) \to \mathcal{L}(\mathcal{H})$, with 
the property that $\tau(1)=1$,
 and which the closed graph theorem implies is continuous. Let $\A$ be the operator norm closed algebra generated by ${\rm ran\ } \tau$. 
Evidently, $\A$ is a commutative Banach algebra with unit,
 containing the components of $T$. We let $\sigma_\A(T)$ denote the algebraic spectrum of $T$ in this algebra $\A$. 
It is well known that $\sigma(T) \subseteq \sigma_\A(T)$ (see, {\em e.g.}, \cite{cur88}). 
Also, by Gelfand Theory, if $X$ denotes the space of complex homomorphisms of $\A$, then $\sigma_{\A}(T) = \set{\chi(T)}{\chi \in X}$. Thus, the proof of Theorem \ref{thm4.20} will be complete if we can prove the following claim.
\begin{claim}\label{claim4.10}
If $\chi \in X$, then $\chi(T) \in K_\delta$.
\end{claim}
To prove this claim we shall require the following simple lemma.
\begin{lemma}\label{lem4.5}
If $\delta$ is utile and $\Omega$ is a neighborhood of $K_\delta$, then there exists a positive constant, $c$, such that
$$\normdelgen{f} \le c\sup_{\lambda \in \Omega}|f(\lambda)|$$
for all $f \in H^\infty(\Omega)$
\end{lemma}
\begin{proof}
The utility of $\delta$ guarantees that the formula $L(f)=f|\gdel$ defines a linear transformation $L:H^\infty(\Omega) \to \hinfdelgen$. The Closed Graph Theorem implies that $L$ is bounded.
\end{proof}
We now turn to the proof of Claim \ref{claim4.10}. Let $\chi \in X$ and set $\alpha=\chi(T)$. As $\chi \circ \tau$ is a continuous complex homomorphism of $O(U)$, it follows from Theorem \ref{thm4.30} that
$$\alpha = \chi (T) = \chi(\tau(\lambda)) = (\chi \circ \tau)(\lambda) \in U.$$
To see that $\alpha$ is in $K_\delta$, we argue by contradiction. If $\alpha \not\in K_\delta$, then there exists $t < 1$ such that $K_\delta \subset G_{t\delta}$ and $\alpha \not\in K_{t\delta}$. As $K_{t\delta}$ is $O(U)$-convex and $\alpha \in U \setminus K_{t\delta}$, there exists a sequence $\{f_k\}$ in $O(U)$ such that
\be\label{eq4.10}
\sup_{\lambda \in K_{t\delta}}|f_k(\lambda)| \le 1
\ee
and
\be\label{eq4.20h}
|f_k(\alpha)| \to \infty.
\ee
Evidently, \eqref{eq4.10} implies via an application of Lemma \ref{lem4.5} that there exists a constant $c$ such that $\normdelgen{f_k} \le c$. But then,
$$|f_k(\alpha)| =|\chi(\tau(f_k))| \le \norm{\tau(f_k)}\le \normdelgen{f_k} \le c$$
for all $k$, contradicting \eqref{eq4.20h}. This completes the proof of Claim \ref{claim4.10}.

\subsection{Strictly bounding implies utile}
\begin{theorem}\label{thm6.10}
If $\delta$ is strictly bounding and $f$ is holomorphic on a neighborhood of $K_\delta  \cap V$, then there exists $F \in \hinfdelgen$ such that $F = f$ on $\gdel \cap V$.
\end{theorem}

\bp Suppose that $\delta$ is strictly bounding and $f$ is holomorphic on a neighborhood of $K_\delta$. Choose $t<1$ so that $\ftdelgen$ is bounded, but sufficiently large so that in addition, $f$ is holomorphic on $\gtdel$.

We first construct a mapping from $t^{-1}\D^m$ into $\C^d$ that serves as a left inverse for $\delta$. Fix $r$. As $\normtdelgen{\lambda^r}  < \infty$, an application of the equivalence of (a) and (d) in Theorem \ref{thm3.10} (with $E=\gtdel$ and $z=\lambda^r$) yields $\Phi^r \in \hinfm$ with
\be\label{eq4.10n}
\lambda^r = \Phi^r(t\delta(\lambda)), \ \ \lambda \in \gdel.
\ee
Define $\Psi^r \in O(t^{-1}\D^m)$ by setting $\Psi^r(z)=\Phi^r(tz)$ for $z \in t^{-1}\D^m$. Finally, amalgamate the functions $\Psi^r$, $r=1,\ldots,d$ into a mapping, $\Psi:t^{-1}\D^m \to \C^d$.

Now, $\delta:\gtdel \to t^{-1}\D^m$ and $\Psi:t^{-1}\D^m \to \C^d$. Furthermore, \eqref{eq4.10n} implies that $\Psi \circ \delta = {\text{id}}_{\gtdel}$. Hence, $\delta$ is 1-1, proper, and unramified. As a consequence, 
not only does $\delta$ embed $\gtdel$ as an analytic submanifold of $t^{-1}\D^m$, but $V=\delta(\gtdel)$ is an analytic variety and $\omega:V \to \C$ is holomorphic on $V$ if and only if $\omega \circ \delta$ is holomorphic on $\gtdel$. As $f$ is holomorphic on $\gtdel$, it follows that,
\be\label{eq4.20}
\omega(\delta(\lambda))=f(\lambda),\ \ \lambda \in \gtdel,
\ee
defines a holomorphic function, $\omega$, on $V$. By the Cartan extension theorem, there exists $F \in O(t^{-1}\D^m)$ such that
\be\label{eq4.30}
F(z)=\omega(z),\ \ z \in V.
\ee

Comparing \eqref{eq4.20} and \eqref{eq4.30} we see that $f(\lambda)=F(\delta(\lambda))$ whenever $\lambda \in \gtdel$. As $F \in O(t^{-1}\d^m)$ implies that $F \in \hinfm$
(this follows from Proposition~\ref{prop4.30} and Theorem~\ref{thm4.20})
 it follows that Condition (d) in Theorem \ref{thm3.10} holds. Hence, applying the equivalence of (a) and (d) in that theorem a second time, we deduce that $f \in \hinfdelgen$, as was to be shown.
\ep

\section{Subordination and Utility}
\label{secsubord}
If $\delta \in O(U)^m$, we say that $\gamma \in O(U)^n$ is an \emph{extension of $\delta$} if $n \ge m$ and $\gamma_l = \delta_l$ for $1 \le l \le m$. If $\gamma$ 
is an extension of $\delta$ and in addition, $K\gamma=K_\delta$, we say that \emph{$\gamma$ is subordinate to $\delta$}. Note that if $\gamma$ is subordinate to $\delta$, then $G_\gamma=G_\delta$, 
$\fgamgen \subseteq \fdelgen$, 
$\normgamgen{\cdot} \le \normdelgen{\cdot}$ and 
$\hinfdelgen \subseteq \hinfgamgen$.

We now describe a procedure from \cite{amy12b} that is very effective for using $\hinfdelgen$ to do function theory. Suppose one is given a p-polyhedron, $K_\delta$ and a function, $f$, holomorphic on a neighborhood of $K_\delta$. As $K_\delta$ is a p-polyhedron, Proposition \ref{prop4.40} will imply that $f \in \hinfdelgen$ provided $\delta$ is bounding. While $\delta$ may not be bounding, there nevertheless exists a bounding subordinate extension $\gamma$ of $\delta$. As $K_\gamma = K_\delta$, the classical function theory has not changed, but now, as $f\in \hinfgamgen$, the structure theory from Sections 2 and 3 of the paper can be applied to $f$. A particularly simple way to construct a bounding subordinate extension $\gamma$ of $\delta$ is to set
\be\label{eq4.40}
\gamma=(\delta_1,\ldots,\delta_m,\epsilon\lambda^1,\ldots,\epsilon\lambda^d)
\ee
where $\epsilon$ is chosen sufficiently small that for each $r=1,\ldots,d$, $|\epsilon\lambda^r| \le 1$ for all $\lambda \in K_\delta$ (recall that $\delta$ is such that $\kdel$ is always compact). Note that if one assumes that $K_\delta \subset \D^m$ (as Oka did), then one can choose $\epsilon =1$.

In \cite{amy12b} the idea in the previous paragraph was used to give a new proof of Oka's Extension Theorem. We give another proof in Section~\ref{sece} below.
%
%
One novel feature of this new argument is that it is valid not only when $\gamma$ is defined by \eqref{eq4.40}, but rather, it merely requires that $\gamma$ is bounding. Another interesting feature of the argument is that it provides bounds for the extension.

When one passes from p-polyhedra to the more general case of analytic polyhedra defined in domains of holomorphy, as in light of Question \ref{ques4.10}, the situation is more complicated. To imitate the argument for p-polyhedra one now needs to know that there are extensions of $\delta$ that are spectrally determining.
\begin{theorem}\label{lem4.10}
If $\delta \in O(U)^m$, then there exist $n \ge m$ and $\gamma \in O(U)^n$ such that $\gamma$ is subordinate to $\delta$ and $\gamma$ is spectrally determining.
\end{theorem}
\begin{proof}
Without loss of generality, we may assume that $\delta$ is bounding and $\normdelgen{\lambda^r} \le 1$ for each $r$. Note that this implies that
if $T \in \fdelgen$, then $\norm{T^r} \le 1$ for $r=1,\ldots,m$. Consequently,
\be\label{eq4.50}
\text{if } \{T_n\} \text{ is a sequence in } \fdelgen \text{, then } \sigma(\oplus_{n=1}^\infty T_n) \subseteq (\D^-)^m.
\ee

We now define $\gamma$. Let $E=(\D^-)^m \setminus U$. (The set $E$ may be empty, in which case
we let $\gamma = \delta$, and skip to the last paragraph of the proof).
By Theorem \ref{thm4.40}, for each $\alpha \in E$ there exist functions $R_\alpha^1,\ldots,R_\alpha^d \in O(U)$ such that
\be\label{eq4.60}
\sum_{r=1}^n R_\alpha^r(\lambda) (\lambda^r - \alpha^r) = 1
\ee
for all $\lambda \in U$. We set
\be\label{eq4.70}
M_\alpha^r = \max_{\lambda \in K_\delta}|R_\alpha^r(\lambda)|
\ee
and let $M_\alpha = (M_\alpha^1,\ldots,M_\alpha^d)$ so that $\norm{M_\alpha}$, the Euclidean norm of $M_\alpha$, is given by
\be\label{eq4.80}
\norm{M_\alpha} = (\sum_{r=1}^d|M_\alpha^r|^2)^{\frac{1}{2}}
\ee
We let $B_\alpha(\epsilon)$ denote the Euclidean ball in $\C^d$ centered at $\alpha$ and fix positive $s<1$. As $E$ is
 compact and $\set{B_\alpha(\frac{s}{\norm{M_\alpha}})}{\alpha \in E}$ is an open cover of $E$, there exist $\alpha_1,\ldots,\alpha_n \in E$ such that
\be\label{eq4.90}
E \subset \bigcup_{i=1}^n B_{\alpha_i}(\frac{s}{\norm{M_{\alpha_i}}}).
\ee
For $r=1,\ldots,d$ and $i=1,\ldots,n$ we define functions $\rho_i^r \in O(U)$ by the formulas,
\be\label{eq4.100}
\rho_i^r = \frac{1}{M_{\alpha_i}^r} R_{\alpha_i}^r
\ee
Finally, define an extension $\gamma$ of $\delta$ by setting
\be\label{eq4.110}
\gamma = (\delta_1,\delta_2,\ldots\delta_m,\ \rho_1^1,\rho_1^2,\ldots\rho_1^d,\ \ldots\ldots\rho_n^1,\rho_n^2,\ \ldots\rho_n^d).
\ee
Note that \eqref{eq4.70} and \eqref{eq4.100} imply that for each $r$ and $i$, $\max_{K_\delta}|\rho_i^r| \le 1$ so that $\gamma$ is subordinate to $\delta$.

We now show that $\gamma$, as constructed above, is spectrally bounding. Accordingly, fix a sequence $\{T_n\}$ in $\fgamgen$ and set $T = \oplus T_n$. We need to show that $\sigma(T) \subseteq K_\delta$.
\begin{claim}\label{claim4.20}
$\sigma(T) \subset U$.
\end{claim}
To prove this claim, observe that \eqref{eq4.50} implies that it suffices to prove that
\be\label{eq4.120}
\beta \in E \implies \beta \not\in \sigma(T).
\ee
To prove \eqref{eq4.120}, fix $\beta \in E$. By \eqref{eq4.90} there exists $i$ such that
\be\label{eq4.130}
\norm{\beta - \alpha_i} \le \frac{s}{\norm{M_{\alpha_i}}}
\ee
If for $S \in \fgamgen$, we define an operator $X(S)$ by the formula,
$$X(S)=\sum_{r=1}^d R_{\alpha_i}^r(S)(\beta^r-\alpha_i^r),$$
then \eqref{eq4.60} implies that
\begin{align}
\sum_{r=1}^d R_{\alpha_i}^r(S)(S^r-\beta^r)
&=\sum_{r=1}^d R_{\alpha_i}^r(S)(S^r-\alpha_i^r)-\sum_{r=1}^d R_{\alpha_i}^r(S)(\beta^r-\alpha_i^r) \notag \\
&=1-X(S).\label{eq4.140}
\end{align}
Also, since by the construction of $\gamma$ and \eqref{eq4.100} we have that
\be\label{eq4.145}
\norm{R_{\alpha_i}^r(S)} \le M_{\alpha_i}^r,
\ee
we see via \eqref{eq4.80} and \eqref{eq4.130} that
\begin{align}
\norm{X(S)} &= \norm{\sum_{r=1}^d R_{\alpha_i}^r(S)(\beta^r-\alpha_i^r)} \notag \\
&\le \sum_{r=1}^d \norm{R_{\alpha_i}^r(S)}|\beta^r-\alpha_i^r| \notag \\
&\le \sum_{r=1}^d M_{\alpha_i}^r|\beta^r-\alpha_i^r| \notag \\
&\le (\sum_{r=1}^d|M_{\alpha_i}^r|^2)^\frac{1}{2} (\sum_{r=1}^d|\beta^r-\alpha_i^r|^2)^\frac{1}{2} \label{eq4.150}\\
&=\norm{M_{\alpha_i}}\norm{\beta-\alpha_i} \notag \\
&\le \norm{M_{\alpha_i}}\frac{s}{\norm{M_{\alpha_i}}} \notag \\
&=s. \notag
\end{align}
We now let $S=T_n$ and direct sum. If we set
$$Y^r = \oplus_{n=1}^\infty R_{\alpha_i}^r(T_n),$$
then \eqref{eq4.145} implies that $Y^r$ is a well defined bounded operator that commutes with $T$. Likewise, if we set
$$X=\oplus_{n=1}^\infty X(T_n),$$
then \eqref{eq4.150} implies that $X$ is a well defined bounded operator that commutes with $T$ with the added property that
\be\label{eq4.160}
\norm{X} \le s <1.
\ee
Finally, observe that by \eqref{eq4.140},
\be\label{eq4.170}
\sum_{r=1}^d Y^r(T^r-\beta^r) = 1-X.
\ee
As \eqref{eq4.160} and \eqref{eq4.170} imply that $\sum_{r=1}^d Y^r(T^r-\beta^r)$ is invertible, it follows that $\beta \not\in \sigma(T)$. This completes the proof of (\ref{eq4.120}) which establishes Claim \ref{claim4.20}.

To complete the proof of Theorem \ref{lem4.10} we need to show that $\sigma(T) \subseteq K_\delta$. As $\sigma(T) \subset U$ and the functions $\delta_l$ are in $O(U)$, 
it follows from the Taylor functional calculus that that 
the operators $\delta_l(T)$ are well defined. Furthermore, by the spectral mapping theorem for analytic functions, $\delta(\sigma(T))=\sigma(\delta(T))$. But $\norm{\delta_l(T)} \le 1$ for each $l$, so that $\sigma(\delta(T)) \subseteq (\D^-)^m$. Hence, $\sigma(T) \subseteq \delta^{-1}((\D^-)^m) = K_\delta$.
\end{proof}

\section{The Oka Extension Theorem on Analytic Polyhedrons}
\label{sece} 

As discussed in \cite{amy12b}, the following result, Theorem \ref{thm5.10} below, represents a refinement of a classical theorem of Oka. However, unlike in \cite{amy12b}, we will not use the 
Oka-Weil Theorem.
\begin{theorem}\label{thm5.10}
Let $U$ be a domain of holomorphy and assume that $\delta_1,\ldots,\delta_m \in O(U)$ with $\delta$ strictly bounding. 
If $\phi$ is holomorphic on a neighborhood of $K_\delta$, then there exists $\Phi$ holomorphic on a neighborhood of $(\D ^m)^-$ such that $\phi(\lambda)=\Phi \circ \delta(\lambda)$ for all $\lambda \in \gdel$. Moreover, $\Phi$ can be chosen so that $\forall\, t < 1$ and sufficiently close to $1$, and $\forall\, \vare > 0$, we have 
\[
\| \Phi \|_{H^\i_{t {\bf m}}}  \ \leq \
(1 + \vare) \| \phi \|_{\hinftdelgen} .
\]
\end{theorem}
\begin{proof}
Let $\phi$ be holomorphic on a neighborhood of $K_\delta$. Choose $t < 1$ such that $\ftdelgen$ is bounded and $\phi$ is holomorphic on $\gtdel$. By Theorem \ref{thm6.10}, $\phi \in \hinftdelgen$. Hence, by Theorem \ref{thm3.10}, there exists $\Psi \in \hinfm$ such that $\phi(\lambda) = \Psi(t\delta(\lambda))$ for all $\lambda \in \gtdel$. If we define $\Phi$ by $\Phi(z)=\Psi(tz)$, then $\Phi$ is holomorphic on a neighborhood of $(\D ^m)^-$ and $\phi(\lambda)=\Phi \circ \delta(\lambda)$ for all $\lambda \in \gdel$.
\end{proof}

We also get a generalization of the Oka-Weil theorem.
\begin{theorem}\label{thm5.20}
Let $U$ be a domain of holomorphy and assume that $\delta_1,\ldots,\delta_m \in O(U)$ with $K_\delta$ bounded in $\C^d$. If $\phi$ is holomorphic on a neighborhood of $K_\delta$, then $\phi$ can be uniformly approximated on $K_\delta$ by holomorphic functions on $U$.
\end{theorem}
\begin{proof}
By Theorem~\ref{lem4.10}, we can 
choose $\gamma \in O(U)^n$ such that $\gamma_l = \delta_l$ for $1 \le l \le m$,
$K_\gamma = \kdel$, and $\gamma$ is spectrally determining. By Theorem~\ref{thm4.20},
$\gamma$ is utile, so $\phi \in H^\i_\gamma$. By Theorem~\ref{thm3.10}, there is a function
$\Phi$ in $H^\i_{\bf n}$ such that $\phi = \Phi \circ \gamma$. The partial sums of the Taylor series of $\Phi$, composed with $\gamma$, give a sequence of polynomials in $\gamma$ that 
converge uniformly to $\phi$ on $\kdel$.
\end{proof}
\section{A Cartan Extension Theorem}
\label{sech}
In this section $U$ is a domain of holomorphy in $\C^d$.
We assume that $V$ is an analytic set in $U$ with the  special presentation:
there exists  $\eta = (\eta_1, \ldots,\eta_n) \in O(U)^n$ such that
\be\label{eq6.10}
V \= \{ \l \in U \ : \ \eta(\l) = 0 \} .
\ee
Any domain of holomorphy $U$ can be exhausted by a sequence  $G_{\d^k} \subset K_{\d^k} \subset G_{\d^{k+1}} \subset U$ of analytic
polyhedra, 
by \cite[Cor. II.3.11]{ran86}. Choose such a sequence.
We make a second assumption about $V$:
\se\att
\begin{eqnarray}
\nonumber
\forall\  \vare > 0, \forall\  k \geq 1, \ & [ F \in H^\i_{\d^{k+1}, {\rm gen}} ] \wedge [ F |_{V \cap G_{\d^{k+1}}} = 0 ]  \\
\Rightarrow  &
\exists\  \phi_r  \in \ H^\i_{\d^{k}, {\rm gen}}\
{\rm s.t.\ }
\| F - \sum_{r=1}^n \eta_r \phi_r \|_{H^\i_{\d^{k},{\rm gen}}}
\leq \vare .
\label{eq7.11}
\end{eqnarray}

Inequality (\ref{eq7.11}) will hold, for example, (once we augment each $\d^k$ 
to make it utile as we can by Theorem~\ref{lem4.10})
if there is some compact set $L \subset U$ with the
property that whenever $\Omega$ is an open set with $ L \subset \Omega \subset U$ and
$F \in O(\Omega)$ vanishes on $ V \cap \Omega$, then $F$ is in the ideal in $O(\Omega)$ generated by $\eta$.

%
%
%

In this section we wish to prove
a special case of the Cartan extension theorem, first proved in \cite{car51};
see \cite[Thm. I.5]{gun3} for a more recent treatment.
It can be argued that the logic is circular, as we use Theorem~\ref{thm6.10}, whose proof
used the Cartan extension theorem. However, whenever the answer to Question~\ref{ques4.10} is known to be yes, such as for a $p$-polyhedron by Proposition~\ref{prop4.30}, 
one does not need Theorem~\ref{thm6.10}.

%
%
%
%
%
%

\bt
\label{thm72}
Let $V \subseteq U$ satisfy (\ref{eq6.10}) and (\ref{eq7.11})
 Suppose $f$ is holomorphic on a neighborhood $\Omega$ of
$V$. Then $f$ has a holomorphic extension to $U$.
\et
\bp 
 By Theorem \ref{lem4.10}, one can assume that each
 $\d^k = (\d^k_1, \dots, \d^k_{n_k})$ is  spectrally determining, and therefore bounding. Moreover, by multiplying $\d^k$ by a number slightly greater than one (and so slightly shrinking
$G_{\d^k}$), one can assume that each $\d^k$ is strictly bounding.

For each $k$, there exists a positive number $M^k$ such that
  $$
\beta^k\  := \ ( \d^k_1, \dots, \d^k_{n_k},  M^k \eta_1,  \dots, M^k \eta_n)$$
will satisfy $K_{\beta^k} \subset \subset \Omega \cap G_{\d^{k+1}}$.
 As $\beta^k$ is an extension of $\delta^k$,
it is {\em a fortiori} strictly bounding, so by Theorem~\ref{thm6.10} it is utile.
Therefore by Theorem~\ref{thm3.10}, we can find a function $\Psi^k$ in $H^\i_{\bf  n_k + n}$ such that
$\Psi^k \circ \beta^k$ is a holomorphic function on $G_{\beta^k}$ that equals
 the extension of
$f$ to $\Omega \cap G_{\d^k}$.

 Now let \[
\Phi^k(\l^1, \dots,  \l^{n_k}) 
\=
\Psi^k(\l^1, \dots, \l^{n_k}, 0, \dots, 0) .
\]
Then $\Phi^k \circ \d^k$ is holomorphic on $G_{\d^k}$, and agrees with $f$ on $V \cap G_{\d^k} $. Define
\[
h^k(\l^1, \dots , \l^d) \= \Phi^k( \d^k(\l)).
\]
Each $h^k$ is in $H^\i_{\d^k, {\rm gen}}$ and provides an extension of $f$ to $G_{\d^k}$; we would like to get an extension on all
of $U$ at once. 

Let us define new functions inductively.
To start, we let $g^1 = h^1, g^2 = h^2$ and $g^{3} = h^{3}$.

 Having defined $g^k$ in $H^\i_{\d^k, {\rm gen}}$ with $$
g^k|_{V \cap G_{\d^k}} \=
f|_{V \cap G_{\d^k}}, $$
observe that $
h^{k+1} - g^k$ is in $H^\i_{\d^k, {\rm gen}}$ and vanishes on 
$V \cap G_{\d^k}$.
So by assumption (\ref{eq7.11}), 
$h^{k+1} - g^k$ can be approximated in $H^\i_{\d^{k-1}, {\rm gen}}$ by functions
of the form $\sum_{r=1}^n \eta_r \phi_r$. By Theorem~\ref{thm3.10}, each
$\phi_r$ in turn can be uniformly approximated on $K_{\d^{k-2}}$ by polynomials
in $\d^{k-1}$, so in particular by functions holomorphic on all of $U$.
Therefore there is a function $q^{k+1}$ in $O(U)$ such that
$q^{k+1} |_V = 0$ and
\be
\label{eq77}
\| h^{k+1} - q^{k+1}  - g^k \|_{H^\i( G_{\d^{k-2}}) } 
\ \leq \ 2^{-k} .
\ee

Define $g^{k+1} = h^{k+1} - q^{k+1} $.
By (\ref{eq77}), the functions $(g^k)$ will be a Cauchy sequence on every compact subset
of $U$, so, by passing to a subsequence, one can extract a limit that is holomorphic on $U$.
Call this limit function $F$. As each $g^{k}$ agrees with $f$ on $V \cap G_{\d^k}$,
we conclude that $F$ is our desired extension of $f$ to $U$.
\ep


%
%

%
\section{Functional Calculus}
\label{secfc}

We shall use formula (\ref{eqc17}) to define a functional calculus for $\hinfdelgen$.
For $\M = \oplus_{l=1}^m \M_l$ a graded Hilbert space, let $P_l$ be the orthogonal projection
onto $\M_l$. If $S = (S^1, \dots , S^m)$ is an $m$-tuple of operators on a Hilbert space
$\h$, we shall let $S_P$ denote the operator on $\h \otimes \M$ given by
\[
S_P \= \sum_{l=1}^m S^l \otimes P^l .
\]
In this notation, $\d(\l)$ from (\ref{eqc16}) becomes $\delta(\l)_P$.
For any vector $\xi \in \M$, define
\[
R_\xi : \h \to \h \otimes \M,\quad v \mapsto v \otimes \xi .
\]

\bd
\label{deffc1}
Let $f$ be a function on $\gdel$ with a $(\d,\gdel)$ realization $(a,\beta,\gamma,D)$ as in
Defintion~\ref{def3.10}. Let $T \in \F_\delta$ be a $d$-tuple of commuting operators
on a Hilbert space $\h$, and assume that $\sigma(T) \subset \gdel$. 
Then we define $f(T)$
by
\be\label{eqfc1}
f(T) \ := \
 a I_\h + R_\beta^* \  \delta(T)_P \
 [ I_h \otimes I_\M -  (I_\h \otimes D) \delta(T)_P]^{-1}\  R_\gamma.
\ee
\ed
Some remarks are in order. 

1. Let $S = (\delta_1(T), \dots , \delta_m(T))$.
The assertion  $\sigma(T) \subset \gdel$ seems to require the calculation of the Taylor spectrum
of $T$. But by the spectral mapping theorem, this becomes the much more innocuous assertion that
\[
\sigma(S^l) \subset \D, \ \forall\  1 \leq l \leq m .
\]

2. To use (\ref{eqfc1}), we need to be able to define $\delta_l(T)$ for each $1 \leq l \leq m$.
If $\delta_l$ is a polynomial or rational function, this is immediate; or if $\delta$ itself comes from
a $(\delta', \gdel)$ realization, we can define $\delta(T)$ by the analogous formula. For general
$\delta_l$, it is still possible to define $\delta(T)$ in some circumstances --- for example if
$T$ is a multiplication operator on a space of holomorphic functions. It is, however, essential to the
definition that we know how to define $\delta(T)$.

3. In order for the right-hand side of (\ref{eqfc1}) to make sense, we need to know that
$ [ I_h \otimes I_\M -  (I_\h \otimes D) S_P]$ is invertible. This follows from Proposition~\ref{propfc1} below.

4. To speak unambiguously of $f(T)$, we need to know that the operator defined
by (\ref{eqfc1}) and the one defined by the Taylor functional calculus agree. This follows from
Theorem~\ref{thmfc2} below.

5. For our notion of functional calculus to be useful, we would like to know that
$\hinfdelgen$ has a rich supply of functions. In Section~\ref{secd} we address the issue
of when every function holomorphic on a neighborhood of $\kdel$ is in $\hinfdelgen$.

\bprop
\label{propfc1}
The operator $I_\h \otimes I_\M -  (I_\h \otimes D) \delta(T)_P$ is invertible
whenever the Taylor spectrum of $T$ is in $\gdel$.
\eprop
\bp
It is sufficient to prove that if
$\sigma(S) \subset \D^m$, then
$\sigma [(I_\h \otimes D) S_P] \subset \D$.
By Lemma~\ref{lemfc1}, there is an invertible operator $A$ so that if $R = A^{-1} S A$,
then each $R^l$ is a strict contraction.
Then
\beq
\lefteqn{
(A^{-1} \otimes I_\M) \left[(I_\h \otimes D)(\sum_{l=1}^m S^l \otimes P^l) \right] (A \otimes I_\M)}
\\
&=& 
(I_\h \otimes D)(\sum_{l=1}^m R^l \otimes P^l ).
\eeq
The latter expression is a product of a strict contraction with a contraction, so its spectrum is
in $\D$. So  as $(I_\h \otimes D)S_P $ is similar to it, the spectrum is the same.
\ep
The following lemma is multivariable version of a well-known result of G.-C. Rota \cite{ro60}.
\bl
\label{lemfc1}
If $\sigma(S) \subset \D^m$, then
there is an invertible operator $A$ so that $\| A^{-1} S^l A \| < 1,\ 1 \leq l \leq m$.
\el
\bp
As each $S^l$ has spectral radius less than $1$, choose $N > 0 $ such that 
$\| (S^l)^N \| \leq C < 1, \ 1 \leq l \leq m$. Define a new norm on $\h$ by
\[
||| v ||| \= \frac{1}{N^m} \sqrt{\sum_p \| p(S) v \|^2} ,
\]
where the sum is over all monomials $p$ of degree at most $N-1$ in each variable.
The norm $||| \cdot |||$ satisfies the parallelogram law, so is a Hilbert space norm.
Moreover, it is similar to the norm $\| \cdot \|$, and
\beq
||| S^l v |||^2 - ||| v |||^2
&\=&
\frac{1}{N^{2m}} \sum_q \| (S^l)^N q(S) v \|^2 - \| q(S) v \|^2 \\
&\leq&
\frac{1}{N^{2m}}\sum_q (C^2 -1) \| q(S) v \|^2  \\
& \leq&
\frac{1}{N^{2m}} (C^2 -1) \| v \|^2
\eeq
where $q$ ranges over monomials that have no positive powers of $z^l$.
So in the $||| \cdot |||$ norm, each $S^l$ is a strict contraction.
Let $A$ be the similarity between the Hilbert space with the $||| \cdot |||$ norm
and the original $\| \cdot \|$ norm.
\ep
\bt
\label{thmfc2}
Assume $f$ and $T$ satisfy the hypotheses of Definition~\ref{deffc1}.
Then the operator $f(T)$ defined be the Taylor functional calculus equals the
operator defined by (\ref{eqfc1}).
\et
\bp
Let $F \in  \ball \hinfm$ be defined by (\ref{eqc23}), so $F(\delta(\l) ) = f(\l)$.
As the Taylor functional calculus respects composition, it is enough to prove
that $F(S)$ defined by the Taylor calculus and $F(S)$ defined by
\be
\label{eqfc3}
a I_\h + R_\beta^* \  S_P \
 [ I_h \otimes I_\M -  (I_\h \otimes D) S_P]^{-1}\  R_\gamma
\ee
agree. 
By the argument in the proof of Proposition~\ref{propfc1}, the contraction
$(I_\h \otimes D) S_P$ has spectral radius less than $1$, so the Neumann series
expansion of (\ref{eqfc3}) converges absolutely. The partial sums form a sequence of
 polynomials in $S$ that converge to $F(S)$ in norm, and by continuity
they also converge to $F(S)$  defined by the Taylor functional calculus.
\ep

\bibliography{../references}.
\end{document}